\documentclass[11pt]{amsart}
\usepackage{amsmath, amssymb,mathabx}
\usepackage[colorlinks=true,linkcolor=blue,urlcolor=blue,citecolor=blue, pagebackref]{hyperref}
\usepackage[alphabetic]{amsrefs}
\usepackage{ color}
\newtheorem{thm}{Theorem}[section]

\newtheorem{lem}[thm]{Lemma}

\theoremstyle{definition}

\numberwithin{equation}{section}

\newcommand\Spec{\operatorname{Spec}}

\newcommand\fl{{fl}}

\newcommand{\op}{\operatorname}
\newcommand\Pic{\operatorname{Pic}}

\newcommand\tensor{\otimes}

\newtheorem{theorem}{Theorem}[section]
\newtheorem{remark}[theorem]{ Remark}

\newtheorem{question}[theorem]{Question}
\newtheorem{proposition}[theorem]{Proposition}

\newtheorem{lemma}[theorem]{Lemma}
\newtheorem{clm}[theorem]{Claim}

\input{xypic}

\begin{document}
\title[Principal bundles on families of curves]{Triviality properties
  of principal bundles on singular curves-II} \author{Prakash Belkale and Najmuddin Fakhruddin}

\begin{abstract}
  For $G$ a split semi-simple group scheme and $P$ a principal
  $G$-bundle on a relative curve $X\to S$, we study a natural
  obstruction for the triviality of $P$ on the complement of a
  relatively ample Cartier divisor $D \subset X$. We show, by
  constructing explicit examples, that the obstruction is nontrivial
  if $G$ is not simply connected but it can be made to vanish, if $S$ is the spectrum of a dvr (and  some other hypotheses),   by a faithfully flat base change. The vanishing of
  this obstruction is shown to be a sufficient condition for etale
  local triviality if $S$ is a smooth curve, and the singular locus of
  $X-D$ is finite over $S$.
\end{abstract}
\maketitle

\section{Introduction}
Let $f:X\to S$ be a proper, flat and finitely presented curve over an
arbitrary scheme (over $\operatorname{Spec}(\Bbb{Z})$) $S$. Let $G$ be
a split reductive group scheme over $\operatorname{Spec}(\Bbb{Z})$,
base changed to $S$, and $B$ a Borel subgroup of $G$. Let $D\subset X$
be a relatively ample effective Cartier divisor which is flat over
$S$, and set ${U}=X\setminus D$. Generalizing results of Drinfeld and
Simpson \cite{DS} for the case of smooth $f$, the following result was
proved in \cite[Theorem 1.4]{BF} without any conditions on the
singularities of $f$:
\begin{theorem}\label{triv3}
  Let $P$ be a principal $G$-bundle on $X$ with $G$ semisimple and
  simply connected. Then, after a surjective \'{e}tale base change
  $S'\to S$, $P$ is trivial on $U_{S'}$.
\end{theorem}

Now suppose $G$ is semisimple, but not necessarily simply connected.
and. Triviality statements similar to the above are proved in
\cite[Theorem 1.5]{BF} but with stronger hypotheses: for example, in
characteristic zero, the Cartier divisor $D$ is not allowed to pass
through the singular locus of $f$, and $D$ is also assumed to be set
theoretically a union of sections of $f$ (and some other mild
conditions).

In this note, motivated in part by the article \cite{Sol}, we study
the analogue of Theorem \ref{triv3} for non-simply connected $G$.  In
this case there is a natural obstruction to local triviality
constructed as follows:

Let $\widetilde{G}$ be the simply connected cover of $G$ and denote by
$\pi_1{(G)}$ the scheme theoretic kernel o the covering map
$\widetilde{G} \to G$. The central exact sequence of sheaves of groups
(on the fppf site of $X$),
$$1\to \pi_1(G)_X \to
\widetilde{G}_X \to {G}_X \to 1$$ gives
rise to a boundary map in  fppf cohomology 
\begin{equation} \label{bound}
H^1_{\fl}(X,G)\to
H^2_{\fl}(X, \pi_1(G)) \, .
\end{equation}

Therefore, from $P\in H^1_{\fl}(X,G)$, we get an element
$\alpha_P\in H^2_{\fl}(X, \pi_1(G))$\footnote{We note that for a
  smooth group scheme, fppf cohomology is the same as etale
  cohomology. In particular, if $|\pi_1(G)|$ is invertible in
  $\mathcal{O}_S$, we may replace fppf cohomology by etale cohomology
  throughout this paper.}. It is clear that if $P$ is trivial on $U$,
then $\alpha_P$ maps to zero in $H^2_{\fl}(U, \pi_1(G))$. Thus, for
the generalization of Theorem \ref{triv3} to hold for $P$, the
following property (L) must hold: \vspace{0.1in}
\begin{enumerate}
\item[(L)]There exists a surjective \'{e}tale morphism $S'\to S$
  such that $\alpha_P$ maps to zero in $H^2_{\fl}(U_{S'}, \pi_1(G))$.
\end{enumerate}
\vspace{0.05in} We show that this property is nontrivial: For
$G=\operatorname{PGL}(m)$, we construct principal $G$-bundles on
families of curves $X\to S$ with nodal singularities, and $D$ passing
through the singularities of $f$, where property (L), so also the
direct generalization of Theorem \ref{triv3}, fails (Proposition
\ref{one}). These examples include cases when $S$ is a smooth curve,
and $X\to S$ is a family of smooth curves degenerating to curve with a
single nodal singularity and the divisor $D$ passes through the
node. Examples for other classical groups $G$ can be constructed using
similar methods.

Even though condition (L) is not always satisfied, we show in Lemma
\ref{lem:ff} that the weaker condition (L$'$) below often holds,
e.g., when $S$ is a smooth curve and $U$ is smooth over $S$:
\vspace{0.1in}
\begin{enumerate}
\item[(L$'$)]There exists a faithfully flat morphism $S'\to S$
  such that $\alpha_P$ maps to zero in $H^2_{\fl}(U_{S'}, \pi_1(G))$.
\end{enumerate}
\vspace{0.05in}
We are thus faced with:
\begin{question}
  Does condition (L$'$) always hold? If so, does there always exist a
  faithfully flat morphism $S' \to S$ such that $P$ becomes trivial on
  $U_{S'}$?
\end{question}
 
We do not know the answer to this in full generality. However, we
prove (Theorem \ref{thm:dim1}):
\begin{theorem} Let $f:X \to S$ be flat projective
  curve, $D \subset X$ a relatively ample Cartier divisor which is
  flat over $S$ and set $U = X \backslash D$. Let $G$ be a semisimple
  group and let $P$ be a principal $G$-bundle on $X$. Assume further
  that
\begin{enumerate}
\item $S$ is an excellent regular (purely) one dimensional scheme and
\item $U$ is smooth over $S$.
\end{enumerate}
Then there is a faithfully flat morphism $S' \to S$ such that $P$
becomes trivial on $U_{S'}$.
\end{theorem}

We also show in Proposition \ref{extensio} that for $G=\op{PGL}(m)$,
and $P$ satisfying a condition weaker than (L), and $S$ smooth, $P$
lifts to a principal $\op{GL}(m)$-bundle on $X$ (after an \'{e}tale
base change in $S$). This result generalizes to arbitrary groups, see
Remark \ref{marigold}.

\section{Consequences of condition (L)}

\subsection{The case $S$ is regular of dimension one}

\begin{lemma} \label{lem:pure} Let $U$ be a regular Noetherian scheme
  and $D \subset U$ a closed subscheme. Let $V = U \backslash D$ and
  let $\{D_i\}$ be the irreducible components of $D$ of codimension
  one.  Then for any integer $n > 0$, the kernel of the restriction
  map $H^2_{\fl}(U, \mu_n) \to H^2_{\fl}(V, \mu_n)$ is the subgroup
  generated by the first Chern classes of all $\mathcal{O}_U(D_i)$.
\end{lemma}

The lemma is well-known for etale cohomology, but we do not know a
reference for fppf cohomology so we give a proof.
\begin{proof}

  We first note that for any scheme $U$, $H^1_{\fl}(U, \mathbb{G}_m)
  = \Pic(U)$, and if $U$ is regular then $H^2_{\fl}(U, \mathbb{G}_m) =
  \mathrm{Br}(U)$.

Consider the commutative diagram
\[
\xymatrix{ H^0_{\fl}(U, \mathbb{G}_m) \ar[r] \ar[d] & H^0_{\fl}(V,
  \mathbb{G}_m) \ar[r] \ar[d] & H^1_{D,\fl}(U, \mathbb{G}_m) \ar[r]^{g_1}
  \ar[d]^{c} & H^1_{\fl}(U, \mathbb{G}_m) \ar[r]^{g_2} \ar[d]^{c_U} &
  H^1_{\fl}(V, \mathbb{G}_m) \ar[d]^{c_V} \\
  H^1_{\fl}(U, \mu_n) \ar[r] \ar[d] & H^1_{\fl}(V, \mu_n) \ar[r]
  \ar[d] & H^2_{D,\fl}(U, \mu_n) \ar[r] \ar[d] & H^2_{\fl}(U, \mu_n)
  \ar[r] \ar[d] & H^2_{\fl}(V, \mu_n)
  \ar[d]  \\
  H^1_{\fl}(U, \mathbb{G}_m) \ar[r]^{f_1} & H^1_{\fl}(V, \mathbb{G}_m)
  \ar[r] & H^2_{D,\fl}(U, \mathbb{G}_m) \ar[r] & H^2_{\fl}(U,
  \mathbb{G}_m) \ar[r]^{f_2} &
  H^2_{\fl}(V, \mathbb{G}_m) \\
}
\]
where the rows come from the long exact sequence of cohomology with
supports and the columns from the Kummer sequence.

The map $f_1$ is surjective because $U$ is regular and the map $f_2$
is injective by \cite[Corollaire 1.10]{GBII}. This implies that $
H^2_{D,\fl}(U, \mathbb{G}_m) = 0$ so the map $c$ is surjective. The
claim then follows by a simple diagram chase, noting that the map
$c_U$ gives, by definition, the first Chern class of a line bundle on
$U$ and the kernel of $g_2$ (equal to the image of $g_1$) is precisely
the subgroup of $\Pic(U)$ generated by the $\mathcal{O}_U(D_i)$.

\end{proof}

\begin{lemma} \label{lem:ff} Let $f: U \to S$ be a smooth
  morphism of relative dimension one with $S$ the spectrum of a
  henselian dvr $R$ with quotient field $K$ and residue field
  $k$. Given any element $\alpha \in H^2_{\fl}(U, \mu_n)$ there
  exists a faithfully flat morphism $S' \to S$, with $S'$ also the
  spectrum of a dvr, such that the pullback of $\alpha$ in $U_{S'}$ is
  $0$.
\end{lemma}

If $R$ is excellent, or $n$ is invertible in $\mathcal{O}_S$, the proof
shows that $S' \to S$ can be chosen to be finite.

\begin{proof}

  Let $U_0$ be the closed fibre of $f$ and set $V = U \backslash U_0$.
  Since $V$ is an affine curve over $K$,
  $H^2_{\fl}(V_{\overline{K}},\mu_n) = 0$. This follows from the
  Kummer sequence, and the following facts for smooth curves $Y$ over an
  algebraically closed field: The vanishing of $\mathrm{Br}(Y)$ (Tsen's theorem)
  and, if $Y$ is affine, the surjectivity of multiplication by $n$ map on
  $\Pic(Y)$ for $n>0$. It follows that there exists a
  finite extension $K_1$ of $K$ so that the image of $\alpha$ in
  $H^2_{\fl}(V,\mu_n)$ becomes $0$ in $H^2_{\fl}(V_{K_1},\mu_n)$.  By
  replacing $R$ by its integral closure $R_1$ in $K_1$ (which is still
  a dvr), functoriality implies that we may assume $\alpha$ is in the
  kernel of the restriction map
  $ r: H^2_{\fl}(U,\mu_n) \to H^2_{\fl}(V,\mu_n) $.

  Since $f$ is smooth, $U$ is regular, so by Lemma \ref{lem:pure} the
  kernel of $r$ is spanned by the fundamental classes of the
  irreducible components of $U_0$. Let $R \to R'$ be a finite map,
  with $R'$ a dvr, such that the ramification degree is divisible by
  $n$ and set $S' = \Spec(R')$. The
  pullbacks of the Chern classes of all the components of $U_0$ become
  divisible by $n$, hence are all $0$ in
  $H^2_{\fl}(U_{S'},\mu_n)$. We conclude that $\alpha$ also becomes
  $0$ in $H^2_{\fl}(U_{S'},\mu_n)$.
\end{proof}


Let $f:X\to S$ and $D$ be as in the introduction, $G$ an
arbitrary semisimple group and $P$ a principal $G$-bundle on $X$.

\begin{proposition}\label{afterT}
Assume
\begin{enumerate}
\item $S$ is an excellent regular (purely) one dimensional scheme
  (e.g., a smooth curve over a field).
\item $X$ is reduced.
\item The closure of the non-regular locus of $U=X-D$ does not
  intersect $D$; this is equivalent to assuming that the non-regular
  locus of $U$ is finite over $S$, e.g., $U$ has isolated
  singularities.
\item $\alpha_P$ is zero when restricted to $U$.
\end{enumerate}
Then there is a surjective \'{e}tale  morphism
$S'\to S$ such that $P$ is trivial on $U_{S'}$.
\end{proposition}
\begin{proof}
  For the sake of clarity we first deal with the case $U$ is regular. Let
  $g:\widetilde{X}\to X$ be a resolution of singularities of $X$
  \cite{Lipman}. By our assumptions $U$ is regular so we can assume
  $U\subset \widetilde{X}$. Note that $\widetilde{X}\to S$ is flat
  since $S$ is regular and one-dimensional.

  We first show that $\widetilde{X}-U$ supports a relatively ample
  Cartier divisor $\widetilde{D}$ (possibly non-effective). By
  assumption $D$ supports a relatively ample divisor $D'$. A
  resolution of singularities for $X$ can be obtained by iterating the
  process of normalization and then blowing up the singular locus
  (cf. \cite{Lipman}, and $X$ is excellent).  Let the resulting
  schemes be denoted by $X_0=X,X_1,\dots, X_s=\widetilde{X}$.  We
  build relatively ample Cartier divisors $D_r$ at each step of this
  resolution $X_r$, $D_0=D'$, and finally set $\widetilde{D}=D_s$. For
  the normalization step, we just pull back the Cartier divisor from
  the previous step. For a resolution step $g:X_{r+1}\to X_r$, let
  $E_{r+1}$ be the exceptional divisor of the blow up $g$. It is easy
  to see that then ${D}_{r+1}=g^*(n{D}_r) - E_{r+1}$ is relatively
  ample for $n$ sufficiently large (use \cite[Proposition
  II.7.10]{Hartshorne}: We may assume that $D'$, and hence each $D_r$,
  is actually ample since $S$ is affine).

Therefore, $\widetilde{X}-U$ supports a relatively ample Cartier
divisor $\widetilde{D}$ (possibly non-effective). Let
$L=\mathcal{O}_{\widetilde{X}}(\widetilde{D})$ be the corresponding
relatively ample line bundle on $\widetilde{X}$; it is trivial on
$U$.

Assume that $S$ is affine.  By replacing $S$ by an \'{e}tale cover we
may assume that $P$ has a $B$-reduction, where $B$ is a Borel subgroup
of $G$, and then we may also assume (as in \S 3.2.2 of \cite{BF}) that
$P$ is induced from an $H$-bundle $E$, where $H$ is a maximal torus of
$B$.

Let $\widetilde{G}$ be the simply connected cover of $G$ and and let
$Z \cong \pi_1(G)$ be the kernel of the covering map
$\widetilde{G} \to G$. Let $\widetilde{H}$ be the maximal torus in
$\widetilde{G}$ mapping onto $H$, so $Z \subset \widetilde{H}$. We
have a commutative diagram 
\[
\xymatrix{
1 \ar[r] & Z \ar[r] \ar[d] &  \widetilde{H} \ar[r] \ar[d] &  H \ar[r] \ar[d]
& 1 \\
1 \ar[r] & Z \ar[r] & \widetilde{G} \ar[r] & G \ar[r]  & 1
}
\]
whose rows are exact sequences of group schemes.  Since we have
assumed that $\alpha_P$ becomes $0$ on $U$, by the commutativity of
the diagram it follows that $E|_U$ lifts to a $\widetilde{H}$-bundle
$\tilde{E}_U$ on $U$. Since $\widetilde{H}$ is a torus and
$\widetilde{X}$ is regular, $\tilde{E}_U$ extends to a
$\widetilde{H}$-bundle $\widetilde{E}$ on $\widetilde{X}$ (this
follows from the fact that line bundles on $U$ extend to
$\widetilde{X}$).

Let $\widetilde{P}$ be the induced $\widetilde{G}$-bundle on
$\widetilde{X}$. Since $\widetilde{G}$ is simply connected, and
$\widetilde{X}-U$ supports a relatively ample Cartier divisor
$\widetilde{D}$, by (almost) the same argument as in \S 3.2.2 of
\cite{BF} (see Remark \ref{improv} below) we see that
$\widetilde{P}|_U$ is trivial, hence $P|_U$ is also trivial.

If $U$ is not regular, we employ the following strategy: We
choose a partial desingularization $\widetilde{X}\to X$ which is an
isomorphism over $U$, such that $\widetilde{X}$ is regular in a
neighborhood of $\widetilde{X}-U$.  To see that such a partial
desingularization exists, first consider a full desingularization
$Q\to X$. We just carry out only those blow ups with support over
${X}-U$, and normalize only in neighbourhoods of inverse images of
$X-U$, and obtain $\widetilde{X}\to X$ which is an isomorphism over
$U$, with $\widetilde{X}$ regular on the complement of $U$.
\end{proof}

\begin{remark}\label{improv}
  \cite[Theorem 1.4]{BF} can be generalized as follows: Let $S$ be an
  arbitrary scheme over $\operatorname{Spec}(\Bbb{Z})$ and let
  $f:X\to S$ be a proper, flat and finitely presented curve over
  $S$. Let $E$ be a principal $G$-bundle on $X$ with $G$ semisimple
  and simply connected. Let $U\subset X$ be an open subset, affine
  over $S$, such that $X-U$ supports a relatively ample Cartier
  divisor $D$ for $X\to S$ (possibly non-effective, whose components
  need not be flat over $S$). Then, after a surjective \'{e}tale base
  change $S'\to S$, $E$ is trivial on $U_{S'}$.

  To prove this we need to slightly modify the proof of
  \cite[Proposition 3.2]{BF} as follows: Assume $S$ is affine. We
  twist by tensor powers of $L=\mathcal{O}(D)$ to find subbundles
  $\mathcal{O}\subset E_i\tensor L^{\tensor{r}}$ with corresponding
  quotients $T_i$, $i=1,2$. On the affine open subset
  $U\subseteq {X}$, these extensions of vector bundles split, and $L$
  is trivial. Therefore, restricted to $U$ we get
  $E_i=\mathcal{O}\oplus T_i$, with $T_i$ a line bundle. But $T_i$ is
  the determinant of $E_i$, hence $T_1$ and $T_2$ are isomorphic on
  $U$.
\end{remark}

By combining Lemma \ref{lem:ff} and Proposition \ref{afterT} we
obtain:

\begin{theorem} \label{thm:dim1} Let $f:X \to S$ be flat projective
  curve, $D \subset X$ a relatively ample Cartier divisor which is
  flat over $S$ and set $U = X \backslash D$. Let $G$ be a semisimple
  group and let $P$ be a principal $G$-bundle on $X$. Assume further
  that
\begin{enumerate}
\item $S$ is an excellent regular (purely) one dimensional scheme and
\item $U$ is smooth over $S$.
\end{enumerate}
Then there is a faithfully flat morphism $S' \to S$ such that $P$
becomes trivial on $U_{S'}$.
\end{theorem}

\begin{proof}
  Since $\pi_1(G)$ is a finite group scheme of multiplicative type, by
  applying Lemma \ref{lem:ff} we may find a faithfully flat finite
  type cover $S_1 \to S$, with $S_1$ also excellent regular and one
  dimensional, such that $\alpha_P$ becomes $0$ on $U_{S_1}$.  We then
  apply Proposition \ref{afterT} to the induced morphism
  $X_1^{red} \to S_1$, where $X_1:= X \times_S S_1$, to get an
  \'{e}tale cover $S' \to S_1$ so that $P$ becomes trivial on
  $U_{S'}$. The composition of the maps $S' \to S_1 \to S$ is the
  desired faithfully flat map

\end{proof}

\subsection{Lifting to vector bundles}
A principal $\op{PGL}(m)$-bundle $P$ on $X$ gives rise to a cohomology
class $\alpha_P\in H^2_{\fl}(X, \mu_m)$ as well a cohomology class
$\beta_P\in H^2_{et}(X, \Bbb{G}_m)$ ($= H^2_{\fl}(X, \Bbb{G}_m)$), by
considering the exact sequence of group schemes
$$1\to \Bbb{G}_{m} \to
{\operatorname{GL}}(m) \to {\operatorname{PGL}}(m)\to 1 \ .$$

Clearly $\alpha_P$ maps to $\beta_P$ under the natural map
$H^2_{\fl}(X, \mu_m)\to H^2_{et}(X, \Bbb{G}_m)$. It is easy to see
 that $\beta_P$ represents the obstruction to lifting $P$
  to a principal ${\operatorname{GL}}(m)$-bundle, i.e., a
  vector bundle on $X$.

  Condition (L) implies that $\beta_P$ maps to zero in
  $H^2_{et}(U, \Bbb{G}_m)$, i.e., $P$ can be lifted to a vector
  bundle on $U$. In fact, under somewhat mild conditions, $P$ can be
  lifted to a vector bundle on $X$ after a surjective \'{e}tale base
  change of $S$:
\begin{proposition}\label{extensio}
  Assume $S$ is smooth, and the smooth locus of $X\to S$ is dense in
  every fiber.  If $\beta_P$ maps to zero in $H^2_{et}(U, \Bbb{G}_m)$,
  then after an \'{e}tale base change in $S$, $P$ comes from a vector
  bundle on $X$ and hence $\beta_P\in H^2_{et}(X, \Bbb{G}_m)$ becomes
  zero.
\end{proposition}
\begin{proof}
  After an \'{e}tale base change in $S$, we can find sections of
  $X\to S$ such that their union is disjoint from $D$ and contained in
  the smooth locus of $X\to S$.  We can also assume that the union of
  these sections is relatively ample. Let $U'\subset X$ be the
  complement of these sections. Using \cite[Theorem 1.5]{BF}, we may
  assume that $P$ is trivial on $U'$ after an \'{e}tale base change in
  $S$. Lift this trivial $\op{PGL}(m)$-bundle to a vector bundle $W$
  on $U'$.

  By assumption, $P$ comes from a vector bundle $V$ on $U$.  Thus, we
  have two $\op{GL}(m)$-bundles $W$ and $V$ on $U'\cap U$ which
  coincide as $\op{PGL}(m)$-bundles. Let $L^*$ be the sheaf of
  isomorphisms $V\to W$ which induce identity on the underlying
  $\op{PGL}(m)$-bundle. Clearly $L^*$ is a $\Bbb{G}_m$-bundle, let $L$
  be the corresponding line bundle. Hence $W$ is isomorphic to
  $V\tensor L$ on $U'\cap U$. Extend $L$ to a line bundle on $U$ ($U$
  is smooth along $U-U'\cap U$ since $S$ is smooth and the sections
  have images in the smooth locus of $X\to S$). Now glue the vector
  bundle $V\tensor L$ (a vector bundle on $U$) with $W$ (a vector
  bundle on $U'$) over $U'\cap U$, to get a vector bundle $A$ on $X$.
  The $\op{PGL}(m)$-bundle induced from $A$ equals $P$ which completes
  the proof.
\end{proof}

The examples in Section \ref{cordes} all had $\beta_P=0$, i.e., came
from vector bundles on $X$. So $\beta_P=0$ is a lot weaker than
$\alpha_P=0$.

\begin{remark}\label{marigold}
  The proof of Prop \ref{extensio} works more generally: For $G$ any
  semisimple group (replacing $\op{PGL}(m)$) let $G'$ be as in Lemma
  \ref{lem:cover} below.  For $P$ a principal $G$-bundle on $X$ we
  get, as above, elements $\alpha_P\in H^2_{\fl}(X, \pi_1(G))$ and
  $\beta_P\in H^2_{et}(X, K)$.  Suppose $\beta_P$ maps to zero in
  $H^2_{et}(U, K)$, which would be the case if $\alpha_P$ maps to zero
  in $H^2_{\fl}(U, \pi_1(G))$ (i.e., if condition (L) holds). Then,
  after an \'{e}tale base change in $S$, $P$ comes from a principal
  $G'$-bundle on $X$ and hence $\beta_P\in H^2_{et}(X, K)$ becomes
  zero.

\end{remark}

The following lemma is well-known, we give a proof for the convenience
of the reader.

\begin{lem} \label{lem:cover} For any semisimple group $G$ there
  exists a reductive group $G'$ mapping surjectively to $G$ with
  kernel a central torus $K$, and such that the derived group of $G'$
  is simply connected.
\end{lem}

\begin{proof}
  Let $\widetilde{G}$ be the simply connected cover of $G$ and let
  $\widetilde{T}$ be any torus in $\widetilde{G}$ containing the
  kernel $Z \cong \pi_1(G)$ of the covering map $\widetilde{G} \to G$.
  Then we may take $G'$ to be
  $(\widetilde{G} \times \widetilde{T}/Z)$, where $Z$ is embedded
  diagonally.  There is a natural map $G' \to G$ induced by projection
  to the first factor and the kernel of this is $\tilde{T}$ (embedded
  in $G'$ via the second factor). Moreover, the derived group of $G'$
  is equal to $\widetilde{G}$ (embedded via the first factor).  One
  gets a somewhat canonical construction by taking $\widetilde{T}$ to
  be a maximal torus
\end{proof}

\begin{remark}
  There are many choices for $G'$, e.g., for $G = \mathrm{PGL}(m)$
  (resp.~$\mathrm{PGSp}(2m)$) one may also take $G'$ to be
  $\mathrm{GL}(m)$ (resp.~$\mathrm{GSp}(2m)$).
\end{remark}

\section{The examples}\label{cordes}

\subsection{}
By constructing examples where property (L) does not hold, we show
that Theorem \ref{triv3} fails if $G$ is not assumed to be simply
connected.

Let $S$ be a smooth curve over a field of characteristic zero (for
simplicity) and let $X \to S$ a family of projective curves with a
unique singular fibre over the point $s_0 \in S$ having an ordinary
double point at the point $x_0$ over $s_0$. Locally in the \'{e}tale
topology at $x_0$, the family looks like the surface with equation
$xy - z^{n+1} = 0$, for some $n >0$, with the map given by
$(x,y,z) \mapsto z$.

We assume that the family has a section $\sigma: S \to X$ not passing
through $x_0$ and we also assume that there is a section
$\tau:S \to X$ with $\tau(s_0) = x_0$. Such families with sections can
be constructed by base changing any family as above by a suitable map
$S' \to S$, with $S'$ also a smooth curve, factoring through $X$. Note
that the local class group at $x_0$ is $\Bbb{Z}/(n+1)\Bbb{Z}$ (use the
method of proof of \cite[Example II.6.5.2]{Hartshorne}), so
$D=(n+1)\tau(S)$ is a Cartier divisor which is flat and, as is easily
seen, relatively ample over $S$. Set $U=X-D$.

Let $L=\mathcal{O}_X(\sigma(S))$; for any positive integer $m$, the
first Chern class of $L$ gives a cohomology class
$c_1(L) \in H_{et}^2(X,\mu_m)$.  Let $P$ be the
$\operatorname{PGL}(m)$ bundle on $X$ induced from the
$\operatorname{GL}(m)$-bundle $L\oplus
\mathcal{O}^{\oplus(m-1)}$.
Using the identification $\pi_1(G)= \mu_m$, it is easy to see that
$\alpha_P=c_1(L)\in H^2_{et}(X, \mu_m)$.

\begin{proposition}\label{one}
  For suitable $m$ and $n$, $P$ is not trivial on $U_{S'}$ for any
  \'{e}tale neighborhood $S'$ of $s_0$.
 \end{proposition}

Let $S'$ be an  \'{e}tale neighbourhood of $s_0\in S$. By functoriality of
the boundary map \eqref{bound} Proposition \ref{one} follows from
\begin{clm}\label{two}
  For suitable $m$ and $n$, the class $c_1(L)$ restricts to a non-zero
  class in $H_{et}^2(X_{S'}-\tau(S'),\mu_m)$.
 \end{clm}
 It suffices to prove this for $S = S'$, since \'{e}tale base change
 does not alter any of the properties of the family $X \to S$.

\subsection{}

The singularity of $X$ at $x_0$ is  \'{e}tale locally equivalent to
$xy -z^{n+1} = 0$, so of type $A_n$. The exceptional divisor of the
minimal resolution $Q$ consists of a chain of $n$ smooth rational
curves $E_1,E_2,\dots, E_n$, each with self-intersection $-2$ and with
$E_i$ intersecting $E_{i+1}$ transversely, $i= 1, \dots, n-1$; see
for example, \cite[III.7]{BPV}. The fibre $F$ over $s_0$, hence its
strict transform $\tilde{F}$, may or may not be irreducible; in the
latter case we write $\tilde{F_1}$ and $\tilde{F_2}$ for the
components. Locally the fibre corresponds to the curve $xy = 0$,
$z=0$, so its strict transform intersects $E_1$ and $E_n$
transversely (in a single point each). The strict transform
$\tilde{D}$ of $D = \tau(S)$ intersects the exceptional divisor in a
single point which must be smooth (since $Q$ is smooth), so lies on a
unique exceptional divisor $E_t$. Note that $t$ can be arbitrary:
if $S'$ is a smooth curve in $Q$ which intersects $E_t$ transversely
at a point $s_0'$ then $X' = S' \times_S X$ has a singular fibre over
$s_0'$ and a section which on the corresponding desingularization $Q'$
passes through $E_t'$.  By considering tangent directions, one sees
that $\tilde{D}$ does not intersect $\tilde{F}$; $F$ and $D$ become
disjoint after a single blowup.

Let $C = \sigma(S)$ and $\tilde{C}$ its strict transform in $Q$. Let
$E = \cup_i E_i$, so we have $X - D = Q- \tilde{D}\cup E$. Note that
$H^2_{ \tilde{D}\cup E, et}(Q,\mu_m)$ is freely generated by the classes of
$\tilde{D}$ and all the $E_i$.  Thus, using the Gysin sequence, it
suffices to show that the class of $\tilde{C}$ in $H_{et}^2(Q,\mu_m)$ is
not in the span of the classes of $\tilde{D}$ and the $E_i$. Assume
that we have an equation
\begin{equation}\label{class} [\tilde{C}] = a [\tilde{D}] + \sum_i
  b_i[E_i]\in H_{et}^2(Q,\mu_m)
\end{equation}
We will get a contradiction, in certain cases, by using elementary
intersection theory. Each irreducible component $G$ of $E \cup \tilde{F}$
is a proper curve, so we have maps
\[
H_{et}^2(Q, \mu_m) \to H_{et}^2(G, \mu_m) \to  \Bbb{Z}/m\Bbb{Z} ,
\]
where the first map is pullback and the second is the degree map
(which is an isomorphism). Moreover, the pullback map is compatible
with the restriction of line bundles or, equivalently, intersections
of divisors.

\subsubsection{}
Suppose $\tilde{F}$ is reducible. We may assume that $\tilde{C}$
passes through $\tilde{F}_1$,  $\tilde{F}_1$ intersects $E_1$
and $\tilde{F}_2$ intersects $E_n$. Restricting both sides of
\eqref{class} to each of $\tilde{F}_1$, $\tilde{F}_2$ and all the
$E_i$ and using the degree isomorphisms, we get $n+2$ equations in
$n+1$ unknowns. We show below that this leads to a contradiction if
$m$ does not divide $t$, which will certainly be the case if $m>n$.

\begin{itemize}
\item Intersecting \eqref{class} with $\tilde{F}_1$, we get
  $b_1=1$. Intersecting with $\tilde{F}_2$ gives $b_n=0$.
\item By induction, we may prove that if $i\leq t$, $b_i=i b_1$. These
  equations are obtained by intersecting the two sides of
  \eqref{class} by $E_1,\dots,E_{t-1}$. Therefore $b_t=tb_1$.
\item By descending induction from $i=n$, we can prove
  that $b_{i}=0$ if $i\geq t$: The case $i=n$ is already
  known. Intersecting with $E_n$ (if $t<n$) gives $b_{n-1}=2b_n$.
  Intersect \eqref{class} with $E_{i}$ (if $i>t$) to get $b_{i-1}
  -2b_i +b_{i+1}=0$, and hence $b_{i-1}=0$.
\end{itemize}
Therefore we have $b_t=tb_1=0\in \Bbb{Z}/m\Bbb{Z}$, and hence $t=0\in
\Bbb{Z}/m\Bbb{Z}$.

\subsubsection{}
Now suppose $\tilde{F}$ is irreducible.

We now intersect both sides of \eqref{class} with the classes of
$\tilde{F}$ and all the $E_i$ and then apply the degree isomorphisms,
getting $n+1$ equations in $n+1$ unknowns. As we show below these
equations imply that $t$ is a linear combination of $n+1$ and
$m$. Therefore, if the gcd of $n+1$ and $m$ does not divide $t$ (for
example, if $n+1$ divides $m$), we reach a contradiction.

\begin{itemize}
\item Intersecting \eqref{class} with $\tilde{F}$, we get
$b_1+b_n=1$.
\item By induction, we may prove that if $i\leq t$, $b_i=i b_1$. These
  equations are obtained by intersecting the two sides of
  \eqref{class} by $E_1,\dots,E_{t-1}$. Therefore $b_t=tb_1$.
\item By descending induction from $i=n$, we can prove that
  that $b_{i}=(n-i+1)b_n$ if $i\geq t$: The case $i=n$ is already
  known. Intersecting with $E_n$ (if $t<n$) gives $b_{n-1}=2b_n$.
  Intersect \eqref{class} with $E_{i}$ (if $i>t$) to get $b_{i-1}
  -2b_i +b_{i+1}=0$, and hence $b_{i-1}=(n-(i-1)-1)b_n$.
\end{itemize}
Therefore $b_t=tb_1= (n-t+1)b_n$, so $t(1-b_n)= (n-t+1)b_n$, hence
$b_n(n+1)=t$. We thus see that $t$ is a linear combination of $n+1$
and $m$.

\subsection{}

Let $X \to S$ be a family of curves as at the beginning of this
section and let $f: Q \to S$ be the minimal desingularization of
$X$. Let $S^{sh}$ be the strict henselisation of $S$ at $s_0$ and
$Q^{sh} = Q \times_S S^{sh}$. By the proper base change theorem, the
restriction map $H_{et}^2(Q^{sh}, \mu_m) = H_{et}^2(X_0, \mu_m)$ is an
isomorphism, where $X_0$ is the (geometric) fibre of $f$ over
$s_0$. Now $H_{et}^2(X_0, \mu_m)$ is a free $\Bbb{Z}/m\Bbb{Z}$-module
with basis the irreducible components of $X_0$, i.e., the $E_i$ and
$\tilde{F}$ (or $\tilde{F}_1$ and $\tilde{F}_2$). It follows from this
that for any $\alpha \in H_{et}^2(Q, \mu_m)$, if the numerical
equations obtained by intersecting both sides of an equality
\begin{equation}
  \label{class2} \alpha = a [\tilde{D}] + \sum_i
  b_i[E_i]\in H_{et}^2(Q,\mu_m)
\end{equation}
with the $E_i$ and $\tilde{F}$ (or $\tilde{F}_1$ and $\tilde{F}_2)$)
can be solved, then in fact \eqref{class2} itself can be solved (for
the pullbacks) in $H_{et}^2(Q^{sh}, \mu_m)$, so also when pulled back
to some \'{e}tale neighbourhood $S'$ of $s_0 \in S$. In particular,
this applies to $[\tilde{C}]$ as above.

\section{Acknowledgements}
We thank P.~ Solis for useful communication
(in 2016).

\vspace{0.05 in}

\noindent
P.B.: Department of Mathematics, University of North Carolina, Chapel Hill, NC 27599, USA\\
{{email: belkale@email.unc.edu}}

\vspace{0.08 cm}

\noindent
N.F.: School of Mathematics, Tata Institute of Fundamental Research, Homi Bhabha Road, Mumbai 400005, India\\
{{email: naf@math.tifr.res.in}}
\vspace{0.08 cm}

\end{document}